\UseRawInputEncoding
\documentclass[final]{siamltex}
\usepackage[leqno]{amsmath}
\usepackage{amssymb}
\usepackage{framed}

\newtheorem{rem}{Remark}

\title{Lagrangian-based methods in convex optimization: prediction-correction frameworks with ergodic  convergence rates\thanks{This work was supported by the National Natural Science Foundation of China under Grant 12171021 and
the Fundamental Research Funds for the Central Universities.}}

\author{Tao Zhang\footnotemark[2] \and Yong Xia\footnotemark[2] \and Shiru Li\footnotemark[2]}

\begin{document}
\maketitle
\renewcommand{\thefootnote}{\fnsymbol{footnote}}
\footnotetext[2]{School of Mathematical Sciences, Beihang University, Beijing, 100191, P. R. China  ({\tt (T. Zhang) shuxuekuangwu@buaa.edu.cn};  {\tt (Y. Xia) yxia@buaa.edu.cn}; {\tt (S. Li, Corresponding author) lishiru@buaa.edu.cn}).}

\begin{abstract}
We study the convergence rates of the classical Lagrangian-based methods and their variants for solving  convex optimization problems with equality constraints. We present a generalized prediction-correction framework to establish $O(1/K^2)$ ergodic convergence rates. Under the strongly convex assumption, based on the presented prediction-correction framework, some Lagrangian-based methods with $O(1/K^2)$ ergodic convergence rates are presented, such as the augmented Lagrangian  method with the indefinite proximal term, the alternating direction method of multipliers (ADMM) with a larger step size up to $(1+\sqrt{5})/2$,
the linearized ADMM with the indefinite proximal term,  and the  multi-block ADMM type method (under an alternative assumption that the gradient of one block is Lipschitz continuous).
\end{abstract}



\pagestyle{myheadings}
\thispagestyle{plain}
\markboth{T ZHANG,  Y XIA,  AND S LI}{ERGODIC CONVERGENCE RATES}
\section{Introduction}

Augmented Lagrangian method (ALM) \cite{hestenes1969multiplier,powell1969method} is a classic and efficient tool for solving the convex programming problem with linear equality constraints:
\begin{equation}\label{P1}\tag{P1}
\min\{ f(x):~ Ax=b\},
\end{equation}
where $f:\mathbb{R}^n\rightarrow\mathbb{R}$ is closed, proper, convex, but not necessarily smooth, $A\in\mathbb{R}^{l\times n}$, and $b\in\mathbb{R}^{l}$. The feasible region of \eqref{P1} is denoted by $\Omega$. Even nowadays, \eqref{P1} has a wide range of applications in compressed sensing, image processing, machine learning and so on.
There are some well-known variants of ALM, for example, the proximal ALM introduced by  Rockafeller \cite{rockafellar1976augmented,rockafellar1976monotone} and the linearized ALM with different special proximity terms.

Large-scale
\eqref{P1} can be efficiently solved if it is equipped with a separable structure. The simplest one reads as
\begin{equation}\label{P2}\tag{P2}
\min \left\{f(x)=f_1(x_1)+f_2(x_2):~ (Ax:=)A_1x_1+A_2x_2=b\right\},
\end{equation}
where $A_1\in\mathbb{R}^{l\times n_1}$, $A_2\in\mathbb{R}^{l\times n_2}$, $n_1+n_2=n$, $b\in\mathbb{R}^{l}$, $f_1:\mathbb{R}^{n_1}\rightarrow\mathbb{R}$ and $f_2:\mathbb{R}^{n_2}\rightarrow\mathbb{R}$ are proper, closed and convex. The alternating direction method of multipliers (ADMM) \cite{gabay1976dual,glowinski1975} is a popular Lagrangian-based approach for solving \eqref{P2}. There are also some variants of ADMM. Following the idea of the proximal ALM, the
proximal ADMM \cite{eckstein1994some} introduces a proper proximal term to make the primal subproblems easier to solve. Moreover, with a carefully selected proximal term, it leads to the linearized ADMM \cite{yang2013linearized}.

As a further extension of \eqref{P2}, we then consider
\begin{equation}\label{P3}\tag{P3}
\min\left\{f(x)=\sum_{i=1}^{m}f_i(x_i):~ (Ax:=)\sum_{i=1}^{m}A_ix_i=b\right\},
\end{equation}
where $m\geq2$, $f_i:\mathbb{R}^{n_i}\rightarrow\mathbb{R}$ is closed proper convex for $i\in[1,2,\cdots,m]$, $A_i\in\mathbb{R}^{l\times n_i}$, $\sum_{i=1}^{m}n_i=n$ and $b\in\mathbb{R}^{l}$,  which remains a special case of \eqref{P1}.
It can be solved by a few ADMM type methods
\cite{he2012alternating,he2017convergence}.

The purpose of this paper is to study the ergodic convergence rates of the
Lagrangian-based methods for solving \eqref{P1}, \eqref{P2} and \eqref{P3}, respectively.

In 2012, He and Yuan \cite{2012On}  presented a unified prediction-correction framework to greatly simplify the traditional heavy analysis on the convergence of the Lagrangian-based methods. Based on their framework, one can easily establish $O(1/K)$ convergence rates of the primal-dual gap in the ergodic sense and point-wise in the non-ergodic sense, see  \cite{he2018my,he2014strictly,he2016convergence,he2020optimally,he2015splitting}. Shefi and Teboulle \cite{shefi2014rate} also presented several Lagrangian-based methods with $O(1/K)$ ergodic convergence rates. In fact, there are plenty  of works on $O(1/K)$ convergence rates in the  ergodic sense, see \cite{chambolle2011first,chambolle2016ergodic,he20121,monteiro2013iteration} and references therein.

To accelerate the ergodic convergence rate,  Chambolle and Pock \cite{chambolle2011first,chambolle2016ergodic} presented a primal-dual algorithm for solving the min-max problem with $O(1/K^2)$ ergodic convergence rate (in terms of the primal-dual gap) under the  strongly convex assumption. Xu \cite{xu2017accelerated} established the same convergence rate of ADMM. It was further extended to  the linearized ADMM by Ouyang et al. \cite{ouyang2015accelerated}. Assuming that the gradient is  Lipschitz continuous rather than requiring the objective function to be strongly convex,
Tian and Yuan \cite{2016An} established $O(1/K^2)$ ergodic convergence rate of ADMM.

{\bf Contributions.}
First, motivated by the unified prediction-correction framework due to He and Yuan \cite{2012On}, we present a generalized prediction-correction framework to achieve $O(1/K^2)$ ergodic convergence rates for the  Lagrangian-based methods by dynamically setting the parameters. For solving \eqref{P1} or \eqref{P2} under the assumption that $f$ or $f_2$ is  strongly convex,
we succeed in establishing $O(1/K^2)$ ergodic convergence rates for the Lagrangian-based methods such as the general linearized  ALM with an indefinite proximal term, ADMM with a larger step size up to $(1+\sqrt{5})/2$  and  the linearized ADMM with an indefinite proximal term.  For solving \eqref{P3}, we obtain $O(1/K^2)$ ergodic convergence rate for the multi-block ADMM type methods under the assumption that  one  item of the objective function is gradient  Lipschitz continuous.

{\bf Outline.}
The remainder of this paper is organized as follows. Section 2 presents the generalized prediction-correction framework with ergodic convergence rates and convergence conditions. Sections 3 establishes $O(1/K^2)$ ergodic convergence rates under the assumption that the objective function is either strongly convex or gradient Lipschitz continuous.
Conclusions are made in Section 4.

{\bf Notation.} Let $x=(x_1,\cdots,x_m)\in \mathbb{R}^n$ be a column vector and $x_i$ be the $i$-th component  or subvector of $x$. The inner product of $x,y$ is given by $\langle x,y\rangle=x^Ty$.
Let $I_n\in\mathbb{R}^{n\times n}$ be the  identity matrix.
Define $\|x\|_D^2= x^TDx$  no matter whether $D$ is positive semidefinite. The Euclidean norm $\|x\|$ is $\|x\|_{I_n}$.
Let $\sigma_{\max}(D)$ and $\sigma_{\min}(D)$ be the maximal  and minimal eigenvalues of $D$, respectively.
Let $\partial f(x)$ be the subdifferential  of the convex function $f(x)$.  Denote by
$\nabla f(x)$   the  gradient of the smooth function $f(x)$.
The following two definitions are standard.
\begin{definition}
	$f:\mathbb{R}^n\rightarrow\mathbb{R}$  is  $\sigma~(\geq0)$-strongly convex if there is  a constant $\sigma\geq0$ such that$$f(y)\geq f(x)+\langle f'(x),y-x\rangle+\frac{\sigma}{2}\|y-x\|^2,~f'(x)\in\partial f(x),~\forall x,y\in\mathbb{R}^n.
	$$
\end{definition}
\begin{definition}
	$f:\mathbb{R}^n\rightarrow\mathbb{R}$  is  $L$-gradient Lipschitz continuous  if $f$ is differentiable and there is a constant $L>0$ such that
	$$f(y)\leq f(x)+\langle \nabla f(x),y-x\rangle+\frac{L}{2}\|y-x\|^2,~\forall x,y\in\mathbb{R}^n.
	$$
\end{definition}
If $f$ is $L$-gradient Lipschitz continuous, we have
\begin{equation}
f(y)\geq f(x)+\langle \nabla f(x),y-x\rangle+\frac{1}{2L}\|\nabla f(x)-\nabla f(y)\|^2,~ \forall x,y\in\mathbb{R}^n.
\end{equation}

\section{The generalized prediction-correction framework with ergodic convergence rates}\label{S2}
We  write in the following the Lagrangian function of \eqref{P1}-\eqref{P3}:
$$L(x,\lambda)=f(x)-\lambda^T(Ax-b),$$
where $\lambda$ is the Lagrange multiplier. We call $(x^*,\lambda^*)$ a saddle point of $L(x,\lambda)$ if it holds that
$$
L(x^*,\lambda)\leq L(x^*,\lambda^*)\leq L(x,\lambda^*),~\forall x\in \mathbb{R}^{n},~\forall\lambda\in\mathbb{R}^{l}.
$$
Throughout this paper, for \eqref{P1}, $x\in\mathbb{R}^{n},x'\in\Omega~{\rm and}~\lambda\in\mathbb{R}^{l}$, we define
\begin{equation}
\label{V2}u=v=\begin{pmatrix}
x\\ \lambda
\end{pmatrix},u'=v'=\begin{pmatrix}
x'\\ \lambda
\end{pmatrix},u^*=v^*=\begin{pmatrix}
x^*\\ \lambda^*
\end{pmatrix}, F(u)=\begin{pmatrix}
-A^T\lambda\\Ax-b
\end{pmatrix}.
\end{equation}
For \eqref{P2}, $x_i\in\mathbb{R}^{n_i} $ ($i=1,2$) and $\lambda\in\mathbb{R}^{l}$, we define
\begin{equation}
\label{V18}
\begin{aligned}
&u=\begin{pmatrix}
x_1\\x_2\\\lambda
\end{pmatrix}, u'=\begin{pmatrix}
x_1'\\x_2'\\\lambda
\end{pmatrix},u^*=\begin{pmatrix}
x_1^*\\x_2^*\\\lambda^*
\end{pmatrix},~F(u)=\begin{pmatrix}
-A_1^T\lambda\\-A^T_2\lambda\\Ax-b
\end{pmatrix},\\&
v=\begin{pmatrix}
x_2\\\lambda
\end{pmatrix},v'=\begin{pmatrix}
x_2'\\\lambda
\end{pmatrix},v^*=\begin{pmatrix}
x_2^*\\\lambda^*
\end{pmatrix}, x'=\begin{pmatrix}
x_1'\\x_2'
\end{pmatrix}\in\Omega.
\end{aligned}
\end{equation}
For \eqref{P3},  $x_i\in\mathbb{R}^{n_i} $ ($i=1,\dots,m$) and $\lambda\in\mathbb{R}^{l}$, we define
\begin{equation}
\begin{aligned}
&\label{D13}
u=\begin{pmatrix}
x_1\\ \vdots\\x_m\\\lambda
\end{pmatrix},u'=\begin{pmatrix}
x_1'\\ \vdots\\x_m'\\\lambda
\end{pmatrix},u^*=\begin{pmatrix}
x_1^*\\ \vdots\\x_m^*\\\lambda^*
\end{pmatrix},F(u)=\begin{pmatrix}
-A_1^T\lambda\\\vdots\\-A_m^T\lambda\\Ax-b
\end{pmatrix},\\&v=\begin{pmatrix}
A_2x_2\\ \vdots\\A_mx_m\\\lambda
\end{pmatrix},v'=\begin{pmatrix}
A_2x_2'\\ \vdots\\A_mx_m'\\\lambda
\end{pmatrix},v^*=\begin{pmatrix}
A_2x_2^*\\ \vdots\\A_mx_m^*\\\lambda^*
\end{pmatrix},
x'=\begin{pmatrix}
x_1'\\ \vdots\\x_m'
\end{pmatrix}\in\Omega.
\end{aligned}
\end{equation}
We will always  use the above definitions of $u,~v$  and  $x',~u',~v'$  for \eqref{P1}-\eqref{P3}, unless explicitly stated otherwise.
As shown in \cite{he2018my,he2014strictly,he2016convergence,2012On}, the saddle point $(x^*,\lambda^*)$ can be alternatively characterized as a solution point of the
following variational inequality (VI):
\begin{equation}
\label{V1}
f(x)-f(x^*)+(u-u^*)^TF(u^*)\geq0,~\forall u\in \mathbb{R}^{n+l}.
\end{equation}

\subsection{He and Yuan's prediction-correction framework}
The following framework due to He and Yuan \cite{2012On} is fundamental in providing convergent algorithms for solving \eqref{V1}, see also \cite{he2014strictly,he2016convergence,he2018class}.
\begin{framed}
	\noindent {\bf[Prediction step.]} With a given $v^k$, find $\widetilde{u}^k$ such that
	\begin{equation}\label{V3} f(x)-f(\widetilde{x}^k)+(u-\widetilde{u}^k)^TF(\widetilde{u}^k)\geq(v-\widetilde{v}^k)^TQ(v^k-\widetilde{v}^k), ~\forall u,
	\end{equation}
	where $Q^T +Q\succeq 0$ (noting that $Q$ is not necessarily symmetric).
	
	\noindent {\bf [Correction step.]} Update $v^{k+1}$ by
	\begin{equation}\label{V4}
	v^{k+1}=v^k-M(v^k-\widetilde{v}^k).
	\end{equation}
\end{framed}

In order to guarantee the convergence of  the algorithms satisfying the above framework, He and Yuan \cite{2012On} present  the following additional assumptions on the  selections of the matrices $Q$ and $M$.
\begin{framed}
	\noindent{\bf[Convergence Condition.]}
	For the matrices $Q$ and $M$ used in \eqref{V3} and \eqref{V4}, there exists a matrix $H
	\succeq0$ such that
	\begin{eqnarray}
	Q&=&	HM, 	\label{V5}\\
	G:=Q^T +&Q&-M^THM	.\label{V6}
	\end{eqnarray}
\end{framed}
We remark that  the condition $Q^T +Q\succ 0$ and $H\succ0$ in the original version of He and Yuan's prediction-correction framework has  been relaxed to $Q^T +Q\succeq 0$  and $H\succeq 0$ without any additional proof.

\begin{lemma} [\cite{he2014strictly,he2016convergence,2012On,he2018class}]\label{L3}
	Under the convergence conditions \eqref{V5}-\eqref{V6},	for the prediction-correction framework \eqref{V3}-\eqref{V4}, we have
	\begin{equation}\label{V36}
	\begin{aligned}
	&f(x)-f(\widetilde{x}^k)+(u-\widetilde{u}^k)^TF(\widetilde{u}^k)
	\\\geq&(v-\widetilde{v}^k)^TQ(v^k-\widetilde{v}^k)
	\\=&\frac{1}{2}(\|v^{k+1}-v\|^2_{H}-\|v^k-v\|^2_{H})+\frac{1}{2}\|v^k-\widetilde{v}^k\|^2_{G},~\forall u.
	\end{aligned}
	\end{equation}
\end{lemma}

We note that the  inequality in \eqref{V36} follows from the prediction step  and the  inequality follows from the  correction step and the  conditions \eqref{V5}-\eqref{V6}. Based on Lemma \ref{L3} and Jensen's inequality, if $G\succeq 0$,   one can  easily  obtain
$O(1/K)$ ergodic convergence rate  by adding up all the inequalities from $k=0$ to $K$.

\subsection{Generalized prediction-correction framework with ergodic convergence rates}\label{S2.2}
In order to pursue a faster convergence rate, we present the following generalized framework with dynamically updated matrices $Q^k$ and $M^k$. Moreover, comparing with prediction step \eqref{V3} in He and Yuan's framework, we need an extra item with a newly introduced variable $z$, see \eqref{V7}. The new variable $z$ will be  set as $x$, $x_i$,  $\nabla f(x_i)$ and so on, see details in the later sections.
\begin{framed}
	\noindent{\bf[Generalized prediction step.]} With a given $v^k$, find $\widetilde{u}^k$ such that
	\begin{equation}\label{V7}\tag{GPS} f(x)-f(\widetilde{x}^k)+(u-\widetilde{u}^k)^TF(\widetilde{u}^k)\geq(v-\widetilde{v}^k)^TQ^k(v^k-\widetilde{v}^k)+\frac{\sigma}{2}\|z^k-z\|^2_R,
	~\forall u,
	\end{equation}
	where $\sigma\geq0$, $R\succeq 0$ and $(Q^k)^T +Q^k\succeq0$ (noting that $Q^k$ is not necessarily symmetric).
	
	\noindent {\bf [Generalized correction step.]} Update $v^{k+1}$ by
	\begin{equation}\label{V8}\tag{GCS}
	v^{k+1}=v^k-M^k(v^k-\widetilde{v}^k).
	\end{equation}
\end{framed}
Correspondingly,
we need the following dynamic sufficient condition to guarantee the convergence.

\begin{framed}
	\noindent{\bf[Generalized convergence condition.]}
	For the matrices $Q^k$ and $M^k$ used in \eqref{V7} and \eqref{V8}, respectively, there exists a matrix $H^k$ such that
	\begin{equation}
	Q^k=H^kM^k,\label{V9}\tag{CC1}
	\end{equation}
	\begin{equation}
	\label{V10}\tag{CC2}G^k:=(Q^k)^T +Q^k-(M^k)^TH^kM^k.
	\end{equation}
\end{framed}
Similar to Lemma \ref{L3},
we obtain the following Lemma \ref{L4} by using the alternative  dynamic convergence conditions \eqref{V9}-\eqref{V10}.
\begin{lemma}\label{L4}
	Under the convergence conditions \eqref{V9}-\eqref{V10}, 	for the generalized 	prediction-correction framework \eqref{V7}-\eqref{V8}, we have
	\begin{equation}
	\label{V11}
	\begin{aligned}
	&f(x)-f(\widetilde{x}^k)+(u-\widetilde{u}^k)^TF(\widetilde{u}^k)
	\\\geq&(v-\widetilde{v}^k)^TQ^k(v^k-\widetilde{v}^k)+\frac{\sigma}{2}\|z^k-z\|^2_R
	\\ =&\frac{1}{2}\left(\|v^{k+1}-v\|^2_{H^k}+\sigma\|z^k-z\|^2_R-\|v^k-v\|^2_{H^k}\right)+\frac{1}{2}\|v^k-\widetilde{v}^k\|^2_{G^k},
	~\forall u.
	\end{aligned}
	\end{equation}
\end{lemma}
\begin{proof}
	The proof of this lemma is similar to Lemma \ref{L4}.
	The inequality in \eqref{V11} follows from the prediction step \eqref{V7}. The  equality follows from the fact that	\begin{eqnarray*}
	&&	(v-\widetilde{v}^k)^TQ^k(v^k-\widetilde{v}^k)\nonumber\overset{\eqref{V8}}{=}(v-\widetilde{v}^k)^TH^k(v^k-{v}^{k+1})\label{G3}\\\nonumber&=&\frac{1}{2}\left(\|v^{k+1}-v\|^2_{H^k}-\|v^{k}-v\|^2_{H^k}+\|v^k-\widetilde{v}^k\|_{H^k}^2-	\|v^{k+1}-\widetilde{v}^k\|_{H^k}^2\right)
	\end{eqnarray*}and
		\begin{eqnarray*}
	&&	\|v^k-\widetilde{v}^k\|_{H^k}^2-	\|v^{k+1}-\widetilde{v}^k\|_{H^k}^2\nonumber\\&=&	\|v^k-\widetilde{v}^k\|_{H^k}^2-	\|(v^{k}-\widetilde{v}^k)-(v^k-{v}^{k+1})\|_{H^k}^2\nonumber\\&\overset{\eqref{V7}}{=}&\|v^k-\widetilde{v}^k\|_{H^k}^2-	\|(v^{k}-\widetilde{v}^k)-M^k(v^k-\widetilde{v}^{k})\|_{H^k}^2\label{G2}\\&=&(v^k-\widetilde{v}^k)^T(2H^kM^k-(M^k)^TH^kM^k)(v^k-\widetilde{v}^k)\nonumber\\&=&(v^k-\widetilde{v}^k)^T((Q^k)^T+Q^k-(M^k)^TH^kM^k)(v^k-\widetilde{v}^k)\nonumber\\&=&\|v^k-\widetilde{v}^k\|_{G^k}^2. \nonumber
	\end{eqnarray*}
\end{proof}

Note that directly adding up \eqref{V11} from $k=0$ to $K$ may fail to establish the convergence except that $H^k\succeq 0$ is a constant matrix and $G^k\succeq 0$.
So we need the following additional convergence condition for our new framework.
\begin{framed}
	\noindent {\bf[Additional convergence condition.]}
	For $r^k>0$ and $H_0^k\succeq0$, it holds that
	\begin{equation}\label{V12}\tag{CC3}
	\begin{aligned}
	&r^k\left(\|v^{k+1}-v' \|^2_{H^k}+\sigma\|z^k-z'\|^2_R-\|v^k-v' \|^2_{H^k}+\|v^k-\widetilde{v}^k\|^2_{G^k}\right)\\& \geq\|v^{k+1}-v'\|^2_{H_0^{k+1}}-\|v^k-v' \|^2_{H_0^k}+\varTheta^{k+1}-\varTheta^k,~\varTheta^k\geq0.
	\end{aligned}
	\end{equation}
\end{framed}
\begin{rem}
	1. We only assume $H^k_0\succeq 0$ rather than $H^k \succeq0$.
	2. The new variable $z'$ allows to be set as $x'$, $x_i'$, or $\nabla f(x_i')$. 3. If $G^k\succeq 0$ and $H^k=H\succeq 0$, then \eqref{V12} holds with $r^k=1$, $H_0^k=H$ and $\varTheta^{k+1}-\varTheta^k=\|v^k-\widetilde{v}^k\|^2_{G^k}\geq0$.\end{rem}

Based on Lemma \ref{L4} and the condition \eqref{V12}, we have the following result.
\begin{theorem}\label{main}
	Under the convergence conditions \eqref{V9}-\eqref{V12}, for the generalized prediction-correction framework \eqref{V7}-\eqref{V8}, we have,
	$$
	f(\widetilde{X}^K)-f(x')-\lambda^T(A\widetilde{X}^K-b) \leq O\left(1 \Big/ \sum_{k=0}^{K}r^k \right),
	$$
	where $\widetilde{X}^K=(\sum_{k=0}^{K}r^k\widetilde{x}^k)/(\sum_{k=0}^{K}r^k)$. In particular, setting $r^k=O(k)$ achieves  $O(1/K^2)$ convergence rate.
\end{theorem}
\begin{proof}
	Multiplying both sides of \eqref{V11} (with $x=x'$, $u=u'$, $v=v'$ and $z=z'$) by $r^k$ and combining condition \eqref{V12}, we obtain
	$$\begin{aligned}
	&r^k[f(x')-f(\widetilde{x}^k)+(u'-\widetilde{u}^k)^TF(\widetilde{u}^k)]\\\geq&\frac{1}{2}(
	\|v^{k+1}-v'\|^2_{H_0^{k+1}}-\|v^k-v' \|^2_{H_0^k}+\varTheta^{k+1}-\varTheta^k
	).\end{aligned}
	$$
	Adding the above inequalities from $k=0$ to $K$, we obtain that
	\begin{equation}\nonumber
	\sum_{k=0}^{K}r^k[f(\widetilde{x}^k)-f(x')-(u'-\widetilde{u}^k)^TF(\widetilde{u}^k)]
	\end{equation}
	is bounded from above.
	Note that
	\begin{equation}\nonumber
	f(x')-f(\widetilde{x}^k)+(u'-\widetilde{u}^k)^TF(\widetilde{u}^k)=f(x')-f(\widetilde{x}^k)+\lambda^T(A\widetilde{x}^k-b).
	\end{equation}
	We complete the proof by Jensen's inequality.
\end{proof}

\section{$O(1/K^2)$ ergodic convergence rates}
Based on Theorem \ref{main}, we  establish $O(1/K^2)$ ergodic convergence rates of a few Lagrangian-based methods
for  solving  \eqref{P1}, \eqref{P2}, and \eqref{P3} under the strongly convex assumption or the gradient Lipschitz continuous assumption.
\subsection{Algorithms for solving \eqref{P1}}
We assume that $f$  is $\sigma~(\geq0)$-strongly convex.
The general proximal ALM (GPALM) for solving \eqref{P1} reads as:
\begin{eqnarray}
({\rm GPALM})~\left\{
\begin{array}{lcl}
x^{k+1}\in\arg\min\limits_{x}\{L_{\beta^k}(x,\lambda^k)+\frac{1}{2}\|x-x^k\|^2_{D^k}\},\\
\lambda^{k+1}=\lambda^{k}-\gamma\beta^k(Ax^{k+1}-b), ~\gamma\in(0,2],
\end{array}
\right.\label{GPALM}
\end{eqnarray}
where $\beta^k>0$, $D^k$ is a symmetric matrix, and $L_{\beta^k}(x,\lambda)$ is the augmented Lagrangian function of  \eqref{P1} defined as
\begin{eqnarray}\label{G1}
L_{\beta^k}(x,\lambda)=f(x)-\lambda^T(Ax-b)+\frac{\beta^k}{2}\|Ax-b\|^2.
\end{eqnarray}
Let $u,~v$ and $F(u)$ be defined  in \eqref{V2}.
Define the artificial vector $\widetilde{v}^k=(\widetilde{x}^k,\widetilde{\lambda}^k)$ with
\begin{equation}\label{V17}
\widetilde{x}^k=x^{k+1},~\widetilde{\lambda}^k=\lambda^k-\beta^k(A\widetilde{x}^k-b).\end{equation}
Using the optimality condition of $x$-subproblem in GPALM \eqref{GPALM}, we obtain
$$
\begin{aligned}
f(x)-&f(x^{k+1})+(x-x^{k+1})^T[-A^T\lambda^k+\beta^kA^T(Ax^{k+1}-b)\\&+D^k(x^{k+1}-x^k)]\geq\frac{\sigma}{2}\|x^{k+1}-x\|^2,~\forall x.
\end{aligned}
$$
Then, by an analysis similar to that in  \cite[Page 10]{he2020optimal}, we can write  GPALM \eqref{GPALM} as:
\begin{framed}
	\noindent{\bf[Prediction step.]}  With a given $v^k$, find $\widetilde{u}^k$ such that
	\begin{equation} 	f(x)-f(\widetilde{x}^k)+(u-\widetilde{u}^k)^TF(\widetilde{u}^k)\geq(v-\widetilde{v}^k)^TQ^k(v^k-\widetilde{v}^k)+\frac{\sigma}{2}\|\widetilde{x}^k-x\|^2, ~\forall u, 	\label{V13}
	\end{equation}
	{\rm where}~
	$Q^k=\begin{pmatrix}
	D^k&0\\
	0&\frac{1}{\beta^k}I_l
	\end{pmatrix}.$	
\end{framed}
According to the dual step in GPALM \eqref{GPALM}, we obtain
\begin{equation}\label{V21}
\lambda^{k+1}=\lambda^{k}-\gamma\beta^k(Ax^{k+1}-b)=\lambda^k-\gamma(\lambda^k-\widetilde{\lambda}^k).\end{equation}
Therefore, $v^{k+1}$ generated by GPALM can be viewed as the output after correcting $\widetilde{v}^k$ by the following scheme:
\begin{framed}
	\noindent {\bf [Correction step.]} Update $v^{k+1}$ by
	\begin{eqnarray}
	&&	v^{k+1}=v^k-M^k(v^k-\widetilde{v}^k),
	~
	M^k=\begin{pmatrix}
	I_n&0\\
	0&\gamma I_l
	\end{pmatrix}.	\label{V14}
	\end{eqnarray}
\end{framed}
Then GPALM \eqref{GPALM}  satisfies \eqref{V7}-\eqref{V8}.
We define \begin{eqnarray}
H^k=\begin{pmatrix}
D^k&0\\
0&\frac{1}{\gamma \beta^k}I_l
\end{pmatrix}, \nonumber
~G^k=\begin{pmatrix}
D^k&0\\
0&\frac{2-\gamma}{\beta^k}I_l
\end{pmatrix}.\label{V23}
\end{eqnarray}
Then $H^k$ and $G^k$ satisfy \eqref{V9}-\eqref{V10}. In order to show Theorem \ref{main}, it suffices to prove the following result.
\begin{theorem}
	Suppose that $f$ is only convex and $D^k= D_0/\beta^k$ with $D_0 \succeq0$. Then $\{v^{k+1}\}$ generated by GPALM \eqref{GPALM} satisfies the convergence condition \eqref{V12} with $r^k=\beta^k$, $\sigma=0$, $\varTheta^k\equiv0$ and
	$H^k_0\equiv\begin{pmatrix}
	D_0&0\\
	0&\frac{1}{\gamma}I_l
	\end{pmatrix}.
	$
\end{theorem}
\begin{proof}
	According to Lemma \ref{L4},
	\eqref{V11} holds.
	Multiplying both sides of \eqref{V11} (with $\sigma=0$)  by $\beta^k$ and then ignoring the nonnegative term	yields that
	\begin{equation}\nonumber
	\begin{aligned} &\beta^k\left(f(x)-f(\widetilde{x}^k)+(u-\widetilde{u}^k)^TF(\widetilde{u}^k)\right)\\
	\geq&\frac{\beta^k}{2}\left(\|v^{k+1}-v\|^2_{H^k}-\|v^k-v\|^2_{H^k}+\|v^k-\widetilde{v}^k\|^2_{G^k}
	\right)\\
	\geq&\frac{1}{2}\left(\|v^{k+1}-v\|^2_{H^{k+1}_0}-\|v^k-v\|^2_{H^k_0}\right).
	\end{aligned}
	\end{equation}
	Setting $x=x'$, $u=u'$ and $v=v'$ in the above inequality completes the proof.~
\end{proof}

Next, we consider the possible indefinite case $D^k\not\succeq 0$. For $\tau\in[\frac{2+\gamma}{4},1]$, we define
\begin{eqnarray}
D^k &:=&\tau r\beta^kI_n-\beta^k A^TA=D^k_1-(1-\tau)\beta^kA^TA,\label{V15} \\
D^k_1&:=& \tau r\beta^kI_n-\tau\beta^k A^TA.\label{V16}
\end{eqnarray}
He et al. \cite{he2020optimal} established $O(1/K)$ ergodic convergence rate for  GPALM \eqref{GPALM} with a fixed $\beta^k$ in the case $D^k\not\succeq 0$. As an extension,   we can obtain $O(1/K^2)$ ergodic convergence rate under the assumption that $f$ is $\sigma~(>0)$-strongly convex.
Setting $x=x'$, $u=u'$, $v=v'$, $z^k=x^{k+1}$ and $z=z'=x'$ in Lemma \ref{L4} yields that
\begin{equation}\label{V38}
\begin{aligned}
&f(x')-f(\widetilde{x}^k)+(u'-\widetilde{u}^k)^TF(\widetilde{u}^k)\\
\geq&\frac{1}{2}\left(\|v^{k+1}-v'\|^2_{H^k}+\sigma\|{x}^{k+1}-x'\|^2-\|v^k-v'\|^2_{H^k}\right)
+\frac{1}{2}\|v^k-\widetilde{v}^k\|^2_{G^k}.
\end{aligned}
\end{equation}
Based on the definition of $G^k$, \eqref{V17} and \eqref{V15}, we have
\begin{equation}\label{V22}
\begin{aligned}
&	\|v^k-\widetilde{v}^k\|^2_{G^k}=\|x^k-x^{k+1}\|^2_{D^k}+\frac{2-\gamma}{\beta^k}\|\lambda^k-\widetilde{\lambda}^{k}\|^2\\=&\|x^k-x^{k+1}\|^2_{D^k_1}+(2-\gamma)\beta^k\|Ax^{k+1}-b\|^2_{}
-(1-\tau)\beta^k\|x^k-x^{k+1}\|^2_{A^TA}\\=&\|x^k-x^{k+1}\|^2_{D^k_1}+(2-\gamma)\beta^k\|x^{k+1}-x'\|^2_{A^TA}
-(1-\tau)\beta^k\|x^k-x^{k+1}\|^2_{A^TA}.
\end{aligned}
\end{equation}
To show  Theorem \ref{main}, it suffices to prove the following result.
\begin{theorem}Suppose that  $f$ is $\sigma~(>0)$-strongly convex.
	Let $D^k$ be given in \eqref{V15}, $r>\|A\|^2$ and \begin{equation}
	\label{V25}
	\beta^k(\tau r\beta^k+\sigma)\geq\tau r(\beta^{k+1})^2,~ \beta^{k+1}\ge\beta^k.
	\end{equation}
	Then $\{v^{k+1}\}$ generated by GPALM \eqref{GPALM} satisfies the convergence condition \eqref{V12} with  $H_0^{k}=\begin{pmatrix}
	\beta^kD_1^k+(1-\tau)(\beta^k)^2A^TA&0\\
	0&\frac{1}{\gamma} I_l
	\end{pmatrix},$
	$r^k=\beta^k$, $R=I_n$,  $z'=x'$, ${z}^k=x^{k+1}$ and $\varTheta^k=0$.
\end{theorem}
\begin{proof}
	Let $$e^k:=\|v^{k+1}-v'\|^2_{H^k}+\sigma\|{x}^{k+1}-x'\|^2-\|v^k-v'\|^2_{H^k}+\|v^k-\widetilde{v}^k\|^2_{G^k}.$$
	According to the structure of $H^k$ and $G^k$ and \eqref{V15}, we obtain
	\begin{equation}
	\begin{aligned}
	e^k&=\|x^{k+1}-x'\|^2_{D_1^k- (1-\tau)\beta^kA^TA+\sigma I_n }+\frac{1}{\gamma\beta^k}\|\lambda^{k+1}-\lambda\|^2\\&-\left(
	\|x^{k}-x'\|^2_{D_1^k- (1-\tau)\beta^kA^TA}+\frac{1}{\gamma\beta^k}\|\lambda^{k}-\lambda\|^2
	\right)+ \|v^k-\widetilde{v}^k\|^2_{G^k}.
	\end{aligned}\label{ER}
	\end{equation}
	Substituting the inequality
	$$-(1-\tau)\beta^k\|x^k-x^{k+1}\|^2_{A^TA}\geq-2(1-\tau)\beta^k\left[\|x^k-x'\|^2_{A^TA}+\|x^{k+1}-x'\|^2_{A^TA}\right]$$
	into \eqref{V22} yields that
	\begin{equation}\label{V31}
	\begin{aligned}
	\|v^k-\widetilde{v}^k\|^2_{G^k}\geq&\|x^k-x^{k+1}\|^2_{D^k_1}+(2\tau-\gamma)\beta^k\|x^{k+1}-x'\|^2_{A^TA}\\&
	-2(1-\tau)\beta^k\|x^k-x'\|^2_{A^TA}.
	\end{aligned}
	\end{equation}
	By combining \eqref{ER} and \eqref{V31}, we obtain
	\begin{equation}\label{3.2.2}
	\begin{aligned}
	e^k\geq \bar{e}^k:=&\|x^{k+1}-x'\|^2_{D_1^k+(1-\tau)\beta^kA^TA+\sigma I_n}+\frac{1}{\gamma\beta^k}\|\lambda^{k+1}-\lambda\|^2\\&-\left(
	\|x^{k}-x'\|^2_{D_1^k+(1-\tau)\beta^kA^TA}+\frac{1}{\gamma\beta^k}\|\lambda^{k}-\lambda\|^2
	\right)\\&+\|x^k-x^{k+1}\|^2_{D^k_1}+(4\tau-\gamma-2)\beta^k\|x^{k+1}-x'\|^2_{A^TA}.
	\end{aligned}
	\end{equation}
	Note that $D_1^k\succ0$ with $r>\|A\|^2$. For any $\tau\in[\frac{2+\gamma}{4},1]$ with $\gamma\in(0,2]$, under the assumption \eqref{V25},
	we have
	\begin{eqnarray}
	&&\beta^k[D_1^k+(1-\tau)\beta^kA^TA+\sigma I_n]\nonumber\\&\overset{\eqref{V16}}{=}&\beta^k[\tau r\beta^kI_n+(1-2\tau)\beta^kA^TA+\sigma I_n]\nonumber\\&\overset{\beta^{k+1}\geq\beta^k}{\succeq}&
	\beta^k[\tau r\beta^kI_n+\sigma I_n]+(1-2\tau)(\beta^{k+1})^2A^TA\label{V24}\\&\overset{\eqref{V25}}{\succeq}&\tau r(\beta^{k+1})^2I_n+(1-2\tau)(\beta^{k+1})^2A^TA\nonumber\\&=&
	\beta^{k+1}[D_1^{k+1}+(1-\tau)\beta^{k+1}A^TA].\nonumber
	\end{eqnarray}
	It follows from \eqref{3.2.2}  and  \eqref{V24}  that
		\begin{equation}\label{V37}
		\begin{aligned}
		\beta^ke^k\geq \beta^k\bar{e}^k
		\overset{\eqref{V24}}{\geq}&\underbrace{\|x^{k+1}-x'\|^2_{\beta^{k+1}D_1^{k+1}+(1-\tau)(\beta^{k+1})^2A^TA }+\frac{1}{\gamma}\|\lambda^{k+1}-\lambda\|^2}_{\|v^{k+1}-v'\|^2_{H_0^{k+1}}}\\
		&-\underbrace{\left(	\|x^{k}-x'\|^2_{\beta^kD_1^k+(1-\tau)(\beta^k)^2A^TA}+\frac{1}{\gamma}\|\lambda^{k}-\lambda\|^2	\right)}_{\|v^{k}-v'\|^2_{H_0^{k}}}\\&+\beta^k\left(\|x^k-x^{k+1}\|^2_{D^k_1}+(4\tau-\gamma-2)\beta^k\|x^{k+1}-x'\|^2_{A^TA}\right).
		\end{aligned}
		\end{equation}
	The proof is complete by
	ignoring some nonnegative terms.
\end{proof}
\begin{rem}\label{rem1}
	The condition \eqref{V25} is satisfied if we set $\beta^k=\delta k$
	for some  $\delta>0$.	
\end{rem}

\begin{rem}\label{rem1:1}
	Adding up \eqref{V37} from $k=0$ to $K$ with the special setting $v'=v^*$, we obtain
	\begin{equation}\nonumber
	\begin{aligned}
	0&\geq\sum_{k=0}^{K}\beta^k\left( L(x^*,\widetilde{\lambda}^k)-L(\widetilde{x}^k,\lambda^*)
	\right)\\&	=\sum_{k=0}^{K}\beta^k\left( f(x^*)-f(\widetilde{x}^k)+(u^*-\widetilde{u}^k)^TF(\widetilde{u}^k)\right)\\& 	\overset{\eqref{V38}}{\geq} \frac{1}{2}\left(
	\|v^{K+1}-v^*\|^2_{H_0^{K+1}}-\|v^0-v^*\|^2_{H_0^0}+\sum_{k=0}^{K}\beta^k\|x^k-x^{k+1}\|^2_{D^k_1}\right).
	\end{aligned}
	\end{equation}
	Note that $\beta^k\|x^k-x^{k+1}\|^2_{D^k_1}=\tau(\beta^k)^2\|x^k-x^{k+1}\|^2_{rI_n-A^TA}$ by the definition of $D_1^k$. It follows from the condition \eqref{V25} that $\beta^k=O(k)$ for $k\rightarrow\infty$. Then we obtain $O(1/k^3)$ convergence rate of $\min_{i}\|x^i-x^{i+1}\|^2$ ($i=0,\dots,k$).
\end{rem}

\subsection{Algorithms for solving \eqref{P2}}
We consider solving \eqref{P2} under the $\sigma~(\geq0)$-strongly convex assumption of $f_2$.
We define the artificial vector $\widetilde{v}^k=(\widetilde{x}_2^k,\widetilde{\lambda}^k)$ with
\begin{equation}
\label{V19}
\widetilde{x}_1^k=x_1^{k+1},~
\widetilde{x}_2^k=x_2^{k+1},~{\rm and}~\widetilde{\lambda}^k=\lambda^k-\beta^k(A_1\widetilde{x}^k_1+A_2{x}^k_2-b).\end{equation}
\subsubsection{ADMM}
The iterative scheme of  ADMM for solving \eqref{P2} reads as:
\begin{eqnarray}\label{ADMM}
({\rm ADMM})~
\begin{cases}
x_1^{k+1}\in\arg\min\limits_{x_1}L_{\beta^k}(x_1,x_2^k,\lambda^k),\\
x_2^{k+1}\in\arg\min\limits_{x_2}L_{\beta^k}(x_1^{k+1},x_2,\lambda^k),\\
\lambda^{k+1}=\lambda^k-\gamma\beta^k(A_1x^{k+1}_1+A_2x^{k+1}_2-b).
\end{cases}
\end{eqnarray}
For the special case $\gamma=1$, Xu \cite{xu2017accelerated} gives $O(1/K^2)$ convergence rate under the strongly convex assumption. Next we consider the general case $\gamma\in (0,(1+\sqrt{5})/2]$  by the prediction-correction framework.

From the optimality condition of $x_2$-subproblem in ADMM \eqref{ADMM} and the $\sigma$-strongly convex assumption of $f_2$, we have
\begin{equation}\label{3.2.4}
\begin{aligned}
&f_2(x_2)-f_2(x_2^{k+1})+(x_2-x_2^{k+1})^T[
-A_2^T\lambda^k\\+&\beta^kA_2^T(A_1x_1^{k+1}+A_2x_2^{k+1}-b)]\geq\frac{\sigma}{2}\|x_2^{k+1}-x_2\|^2,~\forall x_2.
\end{aligned}
\end{equation}
Using an analysis similar to that in \cite[Lemma 3.1]{he2016convergence}, we can write ADMM \eqref{ADMM} in the following framework:
\begin{framed}
	\noindent{\bf[Prediction step.]} With a given $v^k$, find $\widetilde{u}^k$ such that
	\begin{eqnarray*} &&f(x)-f(\widetilde{x}^k)+(u-\widetilde{u}^k)^TF(\widetilde{u}^k)\\
		\geq&&(v-\widetilde{v}^k)^TQ^k(v^k-\widetilde{v}^k)+\frac{\sigma}{2}\|\widetilde{x}_2^k-x_2\|^2, ~ \forall u,
	\end{eqnarray*}
	where $Q^k=\begin{pmatrix}
	\beta^kA_2^TA_2&0\\
	-A_2&\frac{1}{\beta^k}I_l
	\end{pmatrix}.$	
\end{framed}
By \eqref{V19}, we can rewrite the dual update step in ADMM as
\begin{equation}
\label{V26}
\lambda^{k+1}=\lambda^k-\gamma\beta^k(A_1x^{k+1}_1+A_2x^{k+1}_2-b)=\lambda^k-\gamma(\lambda^k-\widetilde{\lambda}^k)+\gamma\beta^kA_2(x_2^k-\widetilde{x}_2^k).\end{equation}
Hence, $v^{k+1}$ generated by ADMM can be viewed as a correction from $\widetilde{v}^k$:
\begin{framed}
	\noindent {\bf [Correction step.]} Update $v^{k+1}$ by
	\begin{eqnarray}
	&&v^{k+1}=v^k-M^k(v^k-\widetilde{v}^k),
	\label{V28}
	~
	M^k=\begin{pmatrix}
	I_{n_2}&0\\
	-\gamma\beta^kA_2&\gamma I_l
	\end{pmatrix}.
	\end{eqnarray}
\end{framed}
Then ADMM \eqref{ADMM}  satisfies \eqref{V7}-\eqref{V8}.
We define
\begin{eqnarray*}
	H^k&=\begin{pmatrix}
		\beta^kA_2^TA_2&0\\
		0&\frac{1}{\gamma \beta^k}I_l
	\end{pmatrix},
	~	G^k=\begin{pmatrix}
		(1-\gamma)\beta^kA_2^TA_2&-(1-\gamma)A_2^T\\
		-(1-\gamma)A_2&\frac{2-\gamma}{\beta^k}I_l
	\end{pmatrix}.
\end{eqnarray*}
Then $H^k$ and $G^k$ satisfy \eqref{V9}-\eqref{V10}.
Based on Lemma \ref{L4},  we have
\begin{equation}\label{3.2.1}
\begin{aligned}
&f(x)-f(\widetilde{x}^k)+(u-\widetilde{u}^k)^TF(\widetilde{u}^k)\\
\geq&\frac{1}{2}\left(\|v^{k+1}-v\|^2_{H^k}+\sigma\|{x}_2^{k+1}-x_2\|^2-\|v^k-v\|^2_{H^k}\right)
+\frac{1}{2}\|v^k-\widetilde{v}^k\|^2_{G^k}.
\end{aligned}
\end{equation}
According to the structure of $G^k$,  we have
\begin{eqnarray}\nonumber
\|v^k-\widetilde{v}^k\|^2_{G^k}=&&(1-\gamma)\beta^k\|x_2^k-x_2^{k+1}\|^2_{A_2^TA_2}+\frac{2-\gamma}{\beta^k}\|\lambda^k-\widetilde{\lambda}^k\|^2\\&&-2(1-\gamma)(x_2^k-x_2^{k+1})^TA_2^T(\lambda^k-\widetilde{\lambda}^k).\label{V30}
\end{eqnarray}
Based on \eqref{V26}, \eqref{V30} and $\lambda^{k+1}$ \eqref{ADMM}, we obtain
	\begin{equation}
	\label{V27}
	\begin{aligned}
	&\|v^k-\widetilde{v}^k\|^2_{G^k}\\=&\beta^k\|A_2(x_2^k-x_2^{k+1})\|^2+\frac{2-\gamma}{\gamma^2\beta^k}\|\lambda^k-\lambda^{k+1}\|^2+\frac{2}{\gamma}(x_2^k-x_2^{k+1})^TA_2^T(\lambda^k-\lambda^{k+1})\\=&\beta^k\|A_2(x_2^k-x_2^{k+1})\|^2+(2-\gamma)\beta^k\|Ax^{k+1}-b\|^2+2\beta^k(x_2^k-x_2^{k+1})^TA_2^T(Ax^{k+1}-b).
	\end{aligned}
	\end{equation}
Then,  the following result immediately guarantees  Theorem \ref{main}.
\begin{theorem}
	Suppose that  $f_2$ is $\sigma~(>0)$-strongly convex. For $\gamma\in (0,(1+\sqrt{5})/2]$ and  $\{\beta^k\}$ satisfying
		\begin{equation}\label{C14}
		\beta^k\left(\beta^k+\frac {\sigma}{\sigma_{\max}(A^T_2A_2)}\right)\geq(\beta^{k+1})^2, ~ \frac{ (\beta^k)^3 \sigma_{\max}(A^T_2A_2)}{\beta^k \sigma_{\max}(A^T_2A_2)+ \sigma } \leq(\beta^{k-1})^2,~\beta^{k+1}\geq\beta^k,
		\end{equation}
	$\{v^{k+1}\}$ generated by ADMM \eqref{ADMM} satisfies the convergence condition \eqref{V12} with
	$
	H_0^k=\begin{pmatrix}
	(\beta^k)^2A_2^TA_2&0\\
	0&\frac{1}{\gamma}I_l
	\end{pmatrix},$
	$r^k=\beta^k$, $R=I_{n_2}$,  $z'=x_2'$, ${z}^k=x_2^{k+1}$ and  $\varTheta^k=(1-\gamma)^2(\beta^{k-1})^2\|Ax^{k}-b\|^2$.
\end{theorem}
\begin{proof}
	According to \eqref{3.2.4}, for any $x_2$, it holds that
	\begin{equation}\label{3.2.5} f_2(x_2)-f_2(x_2^{k})+(x_2-x_2^{k})^T[-A_2^T\lambda^{k-1}+\beta^{k-1}A_2^T(Ax^k-b)]\geq\frac{\sigma}{2}\|x_2^{k}-x_2\|^2.
	\end{equation}
	Putting \eqref{3.2.4} (with $x_2=x_2^k$) and \eqref{3.2.5} (with $x_2=x_2^{k+1}$) together  leads to
	\begin{equation}
	\label{V29}
	\begin{aligned}
	(x_2^k-x_2^{k+1})^TA_2^T[
	\lambda^{k-1}&-\lambda^k-\beta^{k-1}(Ax^{k}-b)+
	\beta^{k}(Ax^{k+1}-b)]\\&\geq{\sigma}\|x_2^{k}-x_2^{k+1}\|^2
	.\end{aligned}
	\end{equation}
	According to the dual update rule, \eqref{V29} is equivalent to
	$$
	\begin{aligned}
	(x_2^k-x_2^{k+1})^TA_2^T[\beta^{k}(Ax^{k+1}-b)-
	(1-\gamma)\beta^{k-1}(Ax^{k}-b)]\geq{\sigma}\|x_2^{k}-x_2^{k+1}\|^2,
	\end{aligned}
	$$
	which can be rearranged as
	\begin{equation}\label{C13}
	\begin{aligned}
	&\beta^k(x_2^k-x_2^{k+1})^TA_2^T(Ax^{k+1}-b)\\&\geq(1-\gamma)\beta^{k-1}(x_2^k-x_2^{k+1})^TA_2^T(Ax^{k}-b)+{\sigma}\|x_2^{k}-x_2^{k+1}\|^2\\&\geq(1-\gamma)\beta^{k-1}(x_2^k-x_2^{k+1})^TA_2^T(Ax^{k}-b)+\sigma'\|A_2(x_2^{k}-x_2^{k+1})\|^2
	,
	\end{aligned}
	\end{equation}
	where $\sigma'= \sigma/\sigma_{\max}(A^T_2A_2)$.
	As $\{\beta^k\}$ is non-decreasing, we have
	\begin{equation}
	\label{C10}
	\begin{aligned}
	&(1-\gamma)\beta^{k-1}(x_2^k-x_2^{k+1})^TA_2^T(Ax^{k}-b)\\&\geq-\frac {(\beta^k+\sigma')}{2}\|A_2(x_2^k-x_2^{k+1})\|^2-\frac{(1-\gamma)^2(\beta^{k-1})^2}{2(\beta^k+\sigma')}\|Ax^{k}-b\|^2
	\\&\geq-\frac {(\beta^k+\sigma')}{2}\|A_2(x_2^k-x_2^{k+1})\|^2-\frac{(1-\gamma)^2(\beta^k)^2}{2(\beta^k+\sigma')}\|Ax^{k}-b\|^2.
	\end{aligned}
	\end{equation}
	Note that  \begin{equation}\label{C12}
	\begin{aligned}
	&-\frac{(\beta^k)^2}{\beta^k+\sigma'}\|Ax^{k}-b\|^2=-\frac{(\beta^k)(\beta^k+\sigma'-\sigma')}{\beta^k+\sigma'}\|Ax^{k}-b\|^2\\=&\beta^k\|Ax^{k+1}-b\|^2-(\beta^k-\frac{\sigma'\beta^k}{\beta^k+\sigma'})\|Ax^{k}-b\|^2-\beta^k\|Ax^{k+1}-b\|^2.
	\end{aligned}
	\end{equation}
	According to \eqref{V27}, \eqref{C13}, \eqref{C10} and \eqref{C12}, we obtain
	\begin{equation}
	\begin{aligned}
	\|v^k-\widetilde{v}^k\|^2_{G^k}&\geq
	(1+\gamma-\gamma^2)\beta^k\|Ax^{k+1}-b\|^2\\+&(1-\gamma)^2
	[\beta^k\|Ax^{k+1}-b\|^2-(\beta^k-\frac{\sigma'\beta^k}{\beta^k+\sigma'})\|Ax^{k}-b\|^2].
	\end{aligned}	
	\end{equation}	
	Then, it follows from $\gamma\in (0,\frac{1+\sqrt{5}}{2}]$ and the condition \eqref{C14} that  \begin{equation}\label{C16}
	\beta^k\|v^k-\widetilde{v}^k\|^2_{G^k}\geq(1-\gamma)^2
	[(\beta^k)^2\|Ax^{k+1}-b\|^2-(\beta^{k-1})^2\|Ax^{k}-b\|^2].
	\end{equation}
	Again, using the condition \eqref{C14}, we have \begin{equation}
	\label{C17}
	\begin{aligned}
	&\beta^k\{
	\|v^{k+1}-v'\|^2_{H^k}+\sigma\|{x}_2^{k+1}-x_2'\|^2-\|v^k-v'\|^2_{H^k}
	\\&\geq\|v^{k+1}-v'\|^2_{H^{k+1}_0}-\|v^k-v'\|^2_{H^k_0}.
	\end{aligned}
	\end{equation}
	Combining \eqref{C16} and \eqref{C17}  completes the proof.
\end{proof}

\subsubsection{Linearized ADMM}
We study the   linearized ADMM (LADMM):
\begin{eqnarray}\label{LADMM}
({\rm LADMM})~
\begin{cases}
x_1^{k+1}=\arg\min\limits_{x_1}L_{\beta^k}(x_1,x_2^k,\lambda^k),\\x_2^{k+1}=\arg\min\limits_{x_2}\{L_{\beta^k}(x_1^{k+1},x_2,\lambda^k)+\frac{1}{2}\|x_2-x_2^k\|_{D^k}^2\},\\
\lambda^{k+1}=\lambda^k-\beta^k(A_1x^{k+1}_1+A_2x^{k+1}_2-b),
\end{cases}
\end{eqnarray}
where
\begin{equation}
D^k=\tau r\beta^kI_{n_2}-\beta^kA_2^TA_2. \label{A15}
\end{equation}
Then $D^k=\tau D_1^k+(\tau-1)\beta^kA_2^TA_2$ with $
D_1^k= r\beta^kI_{n_2}-\beta^kA_2^TA_2.
$
The $O(1/K)$ ergodic convergence rate has been established in
\cite{he2020optimally}  when $\tau\in[3/4,1]$,  $\beta^k\equiv\beta$ and $r>\|A_2\|^2$. Note that $D^k$ could be indefinite so that one can expect more efficient numerical performance. The main result in this subsection is to establish  $O(1/K^2)$ ergodic convergence rate under the strongly convex assumption.

Using the optimality condition of the $x_2$-subproblem in LADMM \eqref{LADMM} and the fact that $f_2$ is $\sigma$-strongly convex, we obtain
\begin{equation}\label{B1}
\begin{aligned}
&f_2(x_2)-f_2(x_2^{k+1})+(x_2-x_2^{k+1})^T[
-A_2^T\lambda^k+\beta^kA_2^T(A_1x_1^{k+1}\\
&+A_2x_2^{k+1}-b)+D^k(x_2^{k+1}-x_2^k)]\geq\frac{\sigma}{2}\|x_2^{k+1}-x_2\|^2, ~\forall x_2.
\end{aligned}
\end{equation}
Based on an analysis similar to that in \cite[Section 3]{he2020optimally}, we can write LADMM in the following framework:
\begin{framed}
	\noindent{\bf[Prediction step.]}With a given $v^k$, find $\widetilde{u}^k$ such that
	\begin{eqnarray}
	&&	f(x)-f(\widetilde{x}^k)+(u-\widetilde{u}^k)^TF(\widetilde{u}^k)\nonumber\\
	\geq&&(v-\widetilde{v}^k)^TQ^k(v^k-\widetilde{v}^k)+\frac{\sigma}{2}\|\widetilde{x}_2^k-x_2\|^2,~\forall u, \nonumber
	\end{eqnarray}
	where
	$Q^k=\begin{pmatrix}
	\tau r\beta^kI_{n_2}&0\\
	-A_2&\frac{1}{\beta^k}I_l
	\end{pmatrix}$.
\end{framed}
\begin{framed}
	\noindent{\bf [Correction step.]} Update $v^{k+1}$ by
	\[
	v^{k+1}=v^k-M^k(v^k-\widetilde{v}^k),~M^k=\begin{pmatrix}
	I_{n_2}&0\\
	-\beta^kA_2& I_l
	\end{pmatrix}.
	\]
\end{framed}
Then ADMM LADMM \eqref{LADMM}  satisfies \eqref{V7}-\eqref{V8}.
We define
\begin{eqnarray*}
	H^k&=\begin{pmatrix}
		\tau r\beta^kI_{n_2}&0\\
		0&\frac{1}{ \beta^k}I_l
	\end{pmatrix},~
	G^k=\begin{pmatrix}
		D^k&0\\
		0&\frac{1}{\beta^k}I_l
	\end{pmatrix},
\end{eqnarray*}
which both satisfy \eqref{V9}-\eqref{V10}.
Since  \[
\lambda^{k+1}=\lambda^{k}-\beta^k(A_1x^{k+1}_1+A_2x^{k+1}_2-b)=\lambda^k-(\lambda^k-\widetilde{\lambda}^k)+\beta^kA_2(x_2^k-\widetilde{x}_2^k),
\]
we have \begin{equation}\label{B13}
\begin{aligned}
&\|v^k-\widetilde{v}^k\|^2_{G^k}
=\|x_2^k-\widetilde{x}_2^{k}\|^2_{D^k}+\frac{1}{\beta^k}\|\lambda^k-\widetilde{\lambda}^k\|^2
\\=&\tau r \beta^k\|x_2^k-x_2^{k+1}\|^2+\frac{1}{\beta^k}\|\lambda^k-\lambda^{k+1}\|^2+2(x_2^k-x_2^{k+1})^TA_2^T(\lambda^{k}-\lambda^{k+1}).
\end{aligned}
\end{equation}
Then  we give the following result.
\begin{theorem}
	Suppose that $f_2$ is $\sigma~(>0)$-strongly convex. For $D^k $ given in $\eqref{A15}$ with $\tau\in [\frac{3}{4},1]$, $r>\|A_2\|^2$ and \begin{equation}\label{A16}
	\beta^k(\tau r\beta^k+\sigma)\geq\tau r(\beta^{k+1})^2,~\beta^{k+1}\ge \beta^k,
	\end{equation} the sequence $\{v^{k+1}\}$ generated by LADMM \eqref{LADMM} satisfies the convergence condition \eqref{V12} with
	$H_0^k=\begin{pmatrix}
	\tau r(\beta^k)^2 I_{n_2}&0\\
	0&I_l
	\end{pmatrix}, $
	$r^k=\beta^k$, $R=I_{n_2}$, $z'=x_2'$, ${z}^k=x_2^{k+1}$ and  $\varTheta^k=\frac{1}{2}\|x_2^{k-1}-x_2^{k}\|^2_{\tau\beta^{k} D^{k}_1+(1-\tau)(\beta^{k})^2A_2^TA_2}$.
\end{theorem}
\begin{proof}
	According to \eqref{B1}, we obtain
	\begin{equation}\label{B2} f_2(x_2)-f_2(x_2^{k})+(x_2-x_2^{k})^T
	[-A_2^T\lambda^{k}+D^{k-1}(x_2^k-x_2^{k-1})]\geq\frac{\sigma}{2}\|x_2^{k}-x_2\|^2,
	~\forall x_2.
	\end{equation}
	Putting \eqref{B1} (with $x_2=x_2^k$) and \eqref{B2} (with $x_2=x_2^{k+1}$) together yields that
	\begin{equation}\label{B7}
	\begin{aligned}
	&(x_2^k-x_2^{k+1})^T[
	A_2^T(\lambda^{k}-\lambda^{k+1})+D^{k}(x_2^{k+1}-x_2^{k})\\
	&-D^{k-1}(x_2^k-x_2^{k-1})]\geq{\sigma}\|x_2^{k}-x_2^{k+1}\|^2.
	\end{aligned}
	\end{equation}
	It follows from $r>\|A_2\|^2$ that $D_1^k\succ0$.
	For $\tau\in[3/4,1]$ and the non-decrease sequence $\{\beta^k\}$,
	by using Cauchy-Schwarz inequality, we have
	\begin{equation}\label{B6}
	\begin{aligned}
	&(x_2^k-x_2^{k+1})^TD^{k-1}(x_2^k-x_2^{k-1})\\
	=&(x_2^k-x_2^{k+1})^T[\tau D_1^{k-1}
	-(1-\tau)\beta^{k-1}A_2^TA_2](x_2^k-x_2^{k-1})
	\\\geq&-\frac{1}{2}\|x_2^k-x_2^{k+1}\|^2_{\tau D^{k-1}_1+(1-\tau)\beta^{k-1}A_2^TA_2}-\frac{1}{2}\|x_2^{k-1}-x_2^{k}\|^2_{\tau D^{k-1}_1+(1-\tau)\beta^{k-1}A_2^TA_2}\\\geq&
	-\frac{1}{2}\|x_2^k-x_2^{k+1}\|^2_{\tau D^{k}_1+(1-\tau)\beta^{k}A_2^TA_2}-\frac{1}{2}\|x_2^{k-1}-x_2^{k}\|^2_{\tau D^{k}_1+(1-\tau)\beta^{k}A_2^TA_2}.
	\end{aligned}
	\end{equation}
	Combining \eqref{B7} and \eqref{B6}, we obtain
	\begin{equation}\label{B3}
	\begin{aligned}
	&(x_2^k-x_2^{k+1})^TA_2^T(\lambda^{k}-\lambda^{k+1})\\
	\geq&(x_2^k-x_2^{k+1})^T[ D^{k}(x_2^{k}-x_2^{k+1})+D^{k-1}(x_2^k-x_2^{k-1})  ]+\sigma\|x_2^{k}-x_2^{k+1}\|^2\\\geq&\|x_2^k-x_2^{k+1}\|^2_{D^k}-\frac{1}{2}\|x_2^k-x_2^{k+1}\|^2_{\tau D^{k}_1+(1-\tau)\beta^kA_2^TA_2}\\&-\frac{1}{2}\|x_2^{k-1}-x_2^{k}\|^2_{\tau D^{k}_1+(1-\tau)\beta^kA_2^TA_2}+\sigma\|x_2^{k}-x_2^{k+1}\|^2\\=&
	\frac{1}{2}\|x_2^k-x_2^{k+1}\|^2_{\tau D^{k}_1+(1-\tau)\beta^kA_2^TA_2+\sigma I_{n_2}}-\frac{1}{2}\|x_2^{k-1}-x_2^{k}\|^2_{\tau D^{k}_1+(1-\tau)\beta^kA_2^TA_2}\\&-2(1-\tau)\beta^k\|x_2^k-x_2^{k+1}\|^2_{A_2^TA_2}
	.
	\end{aligned}
	\end{equation}
	For $\tau\in[3/4,1]$,	using Cauchy-Schwarz inequality again, we obtain\begin{equation}\label{B4}
	\begin{aligned}
	&(x_2^k-x_2^{k+1})^TA_2^T(\lambda^{k}-\lambda^{k+1})\\\geq&-\frac{\beta^k}{2(5-4\tau)}\|x_2^k-x_2^{k+1}\|^2_{A_2^TA_2}-
	\frac{5-4\tau}{2\beta^k}\|\lambda^{k}-\lambda^{k+1}\|^2\\
	\geq&-(\tau-\frac{1}{2})\beta^k\|x_2^k-x_2^{k+1}\|^2_{A_2^TA_2}-
	\frac{5-4\tau}{2\beta^k}\|\lambda^{k}-\lambda^{k+1}\|^2.
	\end{aligned}
	\end{equation}
	Then it follows from \eqref{B13}, \eqref{B3} and \eqref{B4} that
		\begin{equation}\label{B9}
		\begin{aligned}
	&	\|v^k-\widetilde{v}^k\|^2_{G^k}\\\geq&
		\frac{1}{2}\|x_2^k-x_2^{k+1}\|^2_{\tau D^{k}_1+(1-\tau)\beta^kA_2^TA_2+\sigma I_{n_2}}-\frac{1}{2}\|x_2^{k-1}-x_2^{k}\|^2_{\tau D^{k}_1+(1-\tau)\beta^kA_2^TA_2}\\&+\|x_2^k-x_2^{k+1}\|^2_{ \tau D^{k}_1}+2(\tau-\frac{3}{4})\{\beta^k\|  x_2^k-x_2^{k+1}\|^2_{ A_2^TA_2}+\frac{1}{\beta^k}\|\lambda^k-\lambda^{k+1}\|^2\}.
		\end{aligned}
		\end{equation}
	Note that  $\tau\in[3/4,1]$. Then it holds that
	\begin{eqnarray}\nonumber
	\tau D^{k}_1+(1-\tau)\beta^kA_2^TA_2&{=}&\tau r\beta^kI_{n_2}+(1-2\tau)\beta^kA_2^TA_2\\&\overset{\beta^{k+1}\geq\beta^k}{\succeq}& \tau r\beta^kI_{n_2}+(1-2\tau)\beta^{k+1}A_2^TA_2.\label{B29}
	\end{eqnarray}
	Hence, we have
	\begin{eqnarray}
	&&\beta^k[	\tau D^{k}_1+(1-\tau)\beta^kA_2^TA_2+\sigma I_{n_2} ]\nonumber\\&\overset{\eqref{B29},\beta^{k+1}\geq\beta^k}{\succeq}
	& [\tau r(\beta^k)^2+\sigma \beta^k]I_{n_2}+(1-2\tau)(\beta^{k+1})^2A_2^TA_2
	\nonumber\\&\overset{\eqref{A16}}{\succeq}&\tau r(\beta^{k+1})^2I_{n_2}+(1-2\tau)(\beta^{k+1})^2A_2^TA_2\label{B10}\\&=&\tau\beta^{k+1} D^{k+1}_1+(1-\tau)(\beta^{k+1})^2A_2^TA_2.\nonumber
	\end{eqnarray}
	Ignoring some nonnegative terms, multiplying both sides of \eqref{B9}  by $\beta^k$,  and then combining it with \eqref{B10} yields that
	\begin{equation}\label{B11}
	\begin{aligned}
	\beta^k\|v^k-\widetilde{v}^k\|^2_{G^k}\geq&\frac{1}{2}\|x_2^k-x_2^{k+1}\|^2_{\tau\beta^{k+1} D^{k+1}_1+(1-\tau)(\beta^{k+1})^2A_2^TA_2}\\&-\frac{1}{2}\|x_2^{k-1}-x_2^{k}\|^2_{\tau\beta^{k} D^{k}_1+(1-\tau)(\beta^{k})^2A_2^TA_2}.
	\end{aligned}
	\end{equation}
	According to the definition of $H^k$ and the condition \eqref{A16}, we have
	\begin{equation}\label{B12}
	\begin{aligned}
	&\beta^k[\|v^{k+1}-v'\|^2_{H^k}+\sigma\|{x}_2^{k+1}-x_2'\|^2-\|v^k-v'\|^2_{H^k}]\\
	\geq&\|v^{k+1}-v'\|^2_{H^{k+1}_0}-\|v^k-v'\|^2_{H^k_0}.
	\end{aligned}
	\end{equation}
	Combining \eqref{B11} and \eqref{B12} completes the proof.
\end{proof}

\subsection{Algorithms for solving \eqref{P3}}
We consider solving \eqref{P3}. For convenience, we define
\begin{align}
\label{G24}
&&J:=\begin{pmatrix}
I_{l}&0&\cdots&0\\
I _{l}&I _{l}&\cdots&    0    \\
\vdots&\ddots&\ddots&\vdots\\
I_{l}&\cdots&I _{l}&I _{l}\\
\end{pmatrix}\in\mathbb{R}^{(m-1)l\times (m-1)l},～～～
\widetilde{I}:=\begin{pmatrix}
I_l&\cdots &I_l
\end{pmatrix}\in\mathbb{R}^{l\times (m-1)l}.
\end{align}
It has been shown in \cite{chen2016direct} that the  classical ADMM for three blocks may not converge. We consider the following  prediction-correction framework \cite{he2012alternating,he2017convergence} with dynamic settings of $\beta^k$ for solving \eqref{P3}:
\begin{framed}
	\noindent{\bf[Prediction step.]}	With given $x^k$ and $\lambda^k$, find $\widetilde{x}^k$ and $\widetilde{\lambda}^k$ such that
	\begin{equation}\label{D7}
	\begin{cases}
	\widetilde{x}_1^{k}=\arg\min\limits_{x_1}\{L_{\beta^k}({x}_1,x_2^k\cdots,x_m^k,\lambda^k),\\
	\widetilde{x}_j^{k}=\arg\min\limits_{x_j}\{L_{\beta^k}(\widetilde{x}_1^k,\cdots,\widetilde{x}^k_{j-1},x_j,x_{j+1}^k,\cdots,x_m^k,\lambda^k),~j=2,\cdots,m,\\
	\widetilde{\lambda}^k=\lambda^k-\beta^k(A_1\widetilde{x}_1^k+\sum_{j=2}^{m}A_jx_j^k-b).
	\end{cases}
	\end{equation}
\end{framed}
\begin{framed}
	\noindent {\bf [Correction step.]}
	\begin{equation}\label{D9}
	v^{k+1} =v^k-M^k(v^k-\widetilde{v}^k),
	\end{equation}
	where  $M^k=(P^k)^{-T}N^k$,
	$$P^k=\begin{pmatrix}
	\sqrt{\beta^k}	J&0\\
	0&\frac{1}{\sqrt{\beta^k}}I_l
	\end{pmatrix}
	,~
	N^k=\gamma\begin{pmatrix}
	\sqrt{\beta^k}I_{(m-1)l}&0\\
	-\sqrt{\beta^k}\widetilde{I}&\frac{1}{\sqrt{\beta^k}}I_l
	\end{pmatrix},
	~ \gamma\in(0,1]. $$
\end{framed}
Based on  the correction step \eqref{D9}, we have
\begin{eqnarray*}
	\lambda^{k+1}&&=\lambda^k-\gamma\left(\sum^m_{j=2}\beta^k A_j(\widetilde{x}^k_j-x^k_j)+(\lambda^k-\widetilde{\lambda}^k)\right)
	=\lambda^k-\gamma\beta^k(\sum_{i=1}^{m}A_i\widetilde{x}_i^k-b).
\end{eqnarray*}
Assume  that $f_m$ is $L$-gradient Lipschitz continuous.
Then using the optimality condition of  the $x_m$-subproblem in the prediction step \eqref{D7}, we have
\begin{equation}\label{D5}
\begin{aligned}
&f_m(x_m)-f_m(\widetilde{x}_m^{k})+(x_m-\widetilde{x}_m^{k})^T
[-A_m^T\lambda^k+\beta^kA_m^T(\sum_{i=1}^{m}A_i\widetilde{x}_i^k-b)]\\\geq&(x_m-\widetilde{x}_m^{k})^T[\nabla f_m(\widetilde{x}_m^k)-A_m^T\lambda^k+\beta^kA_m^T(\sum_{i=1}^{m}A_i\widetilde{x}_i^k-b)]
\\&+\frac{1}{2L}\|\nabla f_m(\widetilde{x}_m^k)-\nabla f_m({x}_m)\|^2,~\forall x_m,
\end{aligned}
\end{equation}
and
\begin{equation}
\label{D6}
\begin{aligned}
0&=\nabla f_m(\widetilde{x}_m^k)-A_m^T\lambda^k+\beta^kA_m^T(\sum_{i=1}^{m}A_i\widetilde{x}_i^k-b)\\&=\nabla f_m(\widetilde{x}_m^k)+(\frac{1}{\gamma}-1)A_m^T\lambda^k-\frac{1}{\gamma}A_m^T\lambda^{k+1}.
\end{aligned}
\end{equation}
Since $0= \nabla f_m(x_m^*)-A_m^T\lambda^*$, it follows from \eqref{D6} that
\begin{equation}
\begin{aligned}
&\frac{1}{2L}\|\nabla f_m(\widetilde{x}_m^k)-\nabla f_m({x}^*_m)\|^2=\frac{1}{2L}\Big\|A_m^T\left((1-\frac{1}{\gamma})\lambda^k+\frac{1}{\gamma}\lambda^{k+1}
-\lambda^*\right)\Big\|^2 \\ \geq&\frac{\sigma_{\min}(A_mA_m^T)}{2L}\Big\|(1-\frac{1}{\gamma})\lambda^k+\frac{1}{\gamma}\lambda^{k+1}
-\lambda^*\Big\|^2.\label{V32}
\end{aligned}
\end{equation}
Based on an analysis  similar to that in \cite[Theorem 3.3]{he2017convergence}, we can rewrite the prediction step \eqref{D7} as
\begin{framed}
	\noindent{\bf[Prediction step.]} With a given $v^k$, find $\widetilde{u}^k$ such that
	\begin{equation}
	\begin{aligned}
	&f(x)-f(\widetilde{x}^k)+(u-\widetilde{u}^k)^TF(\widetilde{u}^k)\\\geq&(v-\widetilde{v}^k)^TQ^k(v^k-\widetilde{v}^k)+\frac{1}{2L}\|\nabla f_m(\widetilde{x}_m^k)-\nabla f_m({x}_m)\|^2,~\forall u,
	\end{aligned}
	\end{equation}
	where \begin{equation}
	\label{D16}
	Q^k=\begin{pmatrix}
	{\beta^k}J&0\\
	-\widetilde{I}&\frac{1}{{\beta^k}}I_l
	\end{pmatrix}.\end{equation}
\end{framed}

For the prediction-correction framework \eqref{D7}-\eqref{D9}, let $H^k$ and $G^k\succeq 0$  satisfy  $\eqref{V9}$-\eqref{V10}.  Then it holds that
\begin{equation}
\label{F4}
\begin{aligned}
&H^k=Q^k(N^k)^{-1}(P^k)^T=\frac{1}{\gamma}P^k(P^k)^T=\frac{1}{\gamma}\begin{pmatrix}
\beta^kJJ^T&0\\
0&\frac{1}{\beta^k}I_l\\
\end{pmatrix},\\&
G^k=(Q^k)^T+Q^k-(M^k)^TH^kM^k
\\&~~~~=(Q^k)^T+Q^k-\frac{1}{\gamma}(N^k)^TN^k
\overset{(*)}{\succeq} (\frac{1}{\gamma^2}-\frac{1}{\gamma})(N^k)^TN^k\overset{\gamma\in(0,1]}{\succeq}0,
\end{aligned}
\end{equation}
where $(*)$ follows from the fact that
\begin{equation}
\nonumber
\begin{aligned}
(Q^k)^T+Q^k&=\begin{pmatrix}
\beta^k(J^T+J)&-\widetilde{I}^T\\
-\widetilde{I}&\frac{2}{\beta^k}I_l
\end{pmatrix}
\succeq \begin{pmatrix}
\beta^k(J^T+J)&-\widetilde{I}^T\\
-\widetilde{I}&\frac{1}{\beta^k}I_l
\end{pmatrix}=\frac{1}{\gamma^2}(N^k)^TN^k.
\end{aligned}
\end{equation}
Now we have
\begin{equation}
\begin{aligned}
&f(x^*)-f(\widetilde{x}^k)+(u^*-\widetilde{u}^k)^TF(\widetilde{u}^k)\geq
\frac{1}{2}\left(\|v^{k+1}-v^*\|^2_{H^k}-\|v^k-v^*\|^2_{H^k}\right)\\
&+\frac{1}{2}\|v^k-\widetilde{v}^k\|^2_{G^k}+\frac{1}{2L}\|\nabla f_m(\widetilde{x}_m^k)-\nabla f_m({x}^*_m)\|^2.
\end{aligned}
\label{V39}
\end{equation}
Then the following theorem ensures Theorem \ref{main}. For simplicity, we define
\begin{eqnarray}\label{D33}
H_{0}^k&:=&\frac{1}{\gamma}\begin{pmatrix}
JJ^T&0\\
0&\Big(\frac{1}{(\beta^{k})^2}+(1-\gamma)\frac{\sigma_{\min}(A_mA_m^T)}{L\beta^{k}} \Big)I_l	
\end{pmatrix}.
\end{eqnarray}
\begin{theorem}\label{D12}
	Suppose that $f_m$ is $L$-gradient Lipschitz continuous. Let $\{\beta^k\}$ satisfy that
	\begin{equation}\label{D10}
	\frac{1}{(\beta^k)^2}+\frac{\sigma_{\min}(A_mA_m^T)}{ L\beta^k}\geq\frac{1}{(\beta^{k+1})^2}+(1-\gamma)\frac{\sigma_{\min}(A_mA_m^T)}{L\beta^{k+1}}.
	\end{equation}
	Then $\{v^{k+1}\}$ generated by the prediction step \eqref{D7} and the correction step  \eqref{D9} satisfies  the convergence condition \eqref{V12} with
	$r^k=\frac{1}{\beta^k}$, $R=I_{n_m}$,  ${z}^k=\nabla f_m(\widetilde{x}^k_m)$, $\sigma=\frac{\sigma_{\min}(A_mA_m^T)}{L}$, $\varTheta^k\equiv0$, $H_0^k$ defined in \eqref{D33} and the special setting $z'=\nabla f_m(x_m^*),~v'=v^*$.
\end{theorem}
\begin{proof}
	We can verify that
	\begin{equation}
	\label{V33}
	\begin{aligned}
	&\Big	\|(1-\frac{1}{\gamma})\lambda^k+\frac{1}{\gamma}\lambda^{k+1}
	-\lambda^*\Big\|^2\\=&(1-\frac{1}{\gamma})\|\lambda^k-\lambda^*\|^2+\frac{1}{\gamma}\|\lambda^{k+1}-\lambda^*\|^2-(1-\frac{1}{\gamma})\frac{1}{\gamma}\|\lambda^k-\lambda^{k+1}\|^2.
	\end{aligned}
	\end{equation}
	Since  $\gamma\in(0,1]$ and  $\eqref{D10}$, it holds that
	\begin{eqnarray*}
		&&\frac{1}{\beta^k}\{the~right~side~of ~\eqref{V39}\}\\~&\overset{G^k\succeq 0}{\geq}&	\frac{1}{2\beta^k}\Big[\|v^{k+1}-v^*\|^2_{H^k}-\|v^k-v^*\|^2_{H^k}+ \frac{1}{L}\|\nabla f_m(\widetilde{x}_m^k)-\nabla f_m({x}^*_m)\|^2 \Big]\\&\overset{\eqref{V32},\eqref{V33}}{\geq}&
		\frac{1}{2\beta^k}\Big[\|v^{k+1}-v^*\|^2_{H^k}+
		\frac{\sigma_{\min}(A_mA_m^T)}{\gamma L}  \|\lambda^{k+1}-\lambda^*\|^2   \Big]
		\\&&-\frac{1}{2\beta^k}\Big[
		\|v^k-v^*\|^2_{H^k}+	\frac{(1-\gamma)\sigma_{\min}(A_mA_m^T)}{\gamma L}  \|\lambda^{k}-\lambda^*\|^2
		\Big]	\\&\overset{\eqref{D10}}{\geq}&\frac{1}{2}\left[\|v^{k+1}-v^*\|^2_{H_0^{k+1}}-\|v^k-v^*\|^2_{H_0^k}\right].
	\end{eqnarray*}
	The proof is complete.	
\end{proof}

\begin{rem}	
	If $m=2$ and $\gamma=1$, the prediction-correction framework \eqref{D7}-\eqref{D9} reduce to ADMM \eqref{ADMM} with $\gamma=1$. For this special case, the $O(1/K^2)$ ergodic convergence rate under  the gradient Lipschitz continuous assumption has been given in \cite{2016An}. To the best of our knowledge, our general result is the first one for
	the multi-block type ADMM with only one block being gradient Lipschitz continuous.
\end{rem}

Before ending this subsection, we consider solving \eqref{P2} with $f_1(x_1)=g^Tx_1$ by the following proximal ADMM:
\begin{eqnarray*}
	\begin{cases}
		x_1^{k+1}\in\arg\min\limits_{x_1}\{L_{\beta^k}(x_1,x_2^k,\lambda^k)+\frac{1}{2}\|x_1-x_1^k\|^2_{D^k}\},\\
		x_2^{k+1}\in\arg\min\limits_{x_2}L_{\beta^k}(x_1^{k+1},x_2,\lambda^k),\\
		\lambda^{k+1}=\lambda^k-\beta^k(A_1x^{k+1}_1+A_2x^{k+1}_2-b),
	\end{cases}
\end{eqnarray*}
where $D^k\succeq0$. Let  \eqref{V18} and \eqref{V19} hold with $v=(x_1, x_2,\lambda)$. Then $v^k$ satisfies the prediction-correction framework
\eqref{V7}-\eqref{V8} with
$$Q^k=\begin{pmatrix}
D^k&0&0\\
0&\beta^kA_2^TA_2&0\\
0&-A_2&\frac{1}{\beta^k}I_l
\end{pmatrix},~M^k=\begin{pmatrix}
I_{n_1}&0&0\\
0&I_{n_2}&0\\
0&-\beta^kA_2& I_l
\end{pmatrix}. $$
The convergence conditions \eqref{V9}-\eqref{V10} hold with $$
H^k=\begin{pmatrix}
D^k&0&0\\
0&\beta^kA_2^TA_2&0\\
0&0&\frac{1}{\beta^k}I_l
\end{pmatrix}.
$$

If $f_2$  is gradient Lipschitz continuous, the convergence condition \eqref{V12} also holds with $r^k=1/\beta^k$ if we set  $D^k=\beta^kI_{n_1}$. In this case, the $x_1$-subproblem becomes
$$0=g-A_1^T\lambda+\beta^kA_1^T( A_1x^{k+1}_1+A_2x^{k}_2-b )+\beta^k(x_1^{k+1}-x_1^k).
$$
It is worth mentioning  that it suffices to calculate the inverse of $A_1^TA_1+I_{n_1}$ only once in all iterations. We can refer to \cite{li2022alternating} for more discussions on  $O(1/K^2)$ convergence rates in this situation.

If $f_2$ is strongly convex, in order to satisfy the convergence condition \eqref{V12} with $r^k=\beta^k$, we set $D^k=I_{n_1}/\beta^k$.
It turns out that we have to calculate the inverse of $A_1^T A_1+I_{n_1}/(\beta^k)^2$  in every iteration.

\section{Conclusions}
We establish a generalized prediction-correction framework for building ergodic convergence rate. Based on the generalized prediction-correction framework,
we give a few Lagrangian-based methods which enjoy $O(1/K^2)$ ergodic convergence rate under the  assumption of strongly convex or gradient Lipchitz continuous,  such as the general linearized  ALM with an indefinite proximal term, ADMM with a larger step size up to $(1+\sqrt{5})/2$,  the linearized ADMM with an indefinite proximal term and multi-block ADMM type method. Building the non-ergodic convergence rate from prediction-correction framework is the future work.

\bibliographystyle{siamplain}
\bibliography{ref}

\end{document}